\documentclass[12pt]{amsart}
\usepackage{amsmath, amssymb}
\newtheorem{theorem}{\sc Theorem}[section]
\newtheorem{lemma}[theorem]{\sc Lemma}
\newtheorem{proposition}[theorem]{\sc Proposition}

\newcommand{\T}{T}
\newcommand{\pr }{\mathrm{Pr} }

\begin{document}
\title[On groups with BFC-covered word values]{On groups with BFC-covered word values}
\thanks{The first and second authors are members of GNSAGA (INDAM), 
and the third author was  supported by  FAPDF and CNPq.}

\author{Eloisa Detomi}
\address{Dipartimento di Matematica \lq\lq Tullio Levi-Civita\rq\rq, Universit\`a di Padova, Via Trieste 63, 35121 Padova, Italy} 
\email{eloisa.detomi@unipd.it}
\author{Marta Morigi}
\address{Dipartimento di Matematica, Universit\`a di Bologna\\
Piazza di Porta San Donato 5 \\ 40126 Bologna \\ Italy}
\email{marta.morigi@unibo.it}
\author{Pavel Shumyatsky}
\address{Department of Mathematics, University of Brasilia\\
Brasilia-DF \\ 70910-900 Brazil}
\email{pavel@unb.br}


\begin{abstract} 
For a group $G$ and a positive integer $n$ write $B_n(G) = \{x \in G : |x^G | \le n\}$. If $s\geq1$ and $w$ is a group word, say that $G$ satisfies the 
$(n,s)$-covering condition with respect to the word $w$ if there exists a subset $S\subseteq G$ such that $|S| \le s$ and all $w$-values of $G$ are contained in $B_n(G)S$. In a natural way, this condition emerged in the study of probabilistically nilpotent groups of class two. In this paper we obtain the following results.
\bigskip

\noindent{\it Let $w$ be a multilinear commutator word on $k$ variables and let $G$ be a group satisfying the $(n,s)$-covering condition with respect to the word $w$. Then $G$ has a soluble subgroup $\T$ such that $[G : \T]$ and the derived length of $\T$ are both $(k,n,s)$-bounded.} (Theorem \ref{soluble-main}.)
\bigskip

\noindent{\it Let $k\geq1$ and $G$ be a group satisfying the $(n,s)$-covering condition with respect to the word $\gamma_k$. Then (1)
$\gamma_{2k-1}(G)$ has a subgroup $\T$ such that $[\gamma_{2k-1}(G) : \T]$ and $|\T'|$ are both $(k,n,s)$-bounded; and (2) $G$ has a nilpotent subgroup $U$ such that $[G : U]$ and the nilpotency class of $U$ are both $(k,n,s)$-bounded.} (Theorem \ref{cov-nilp}.)

 \end{abstract}
\maketitle
\section{Introduction}

An interesting covering condition has emerged in the study of probabilistically nilpotent groups of class two \cite{ebeshu}. If $G$ is a finite group,  we set 
$$d_k(G) = |\{(x_1,...,x_{k+1}) \in G^{k+1} : [x_1,...,x_{k+1}] = 1\}|/|G|^{k+1}, $$
 and we say that $G$ is probabilistically nilpotent of class $k$ if $d_k(G)$ is bounded away from zero. 

The well-known theorem of P. M. Neumann \cite{pmneumann} says that if $G$ is probabilistically nilpotent of class $1,$ i.e. 
 $d_1(G)\geq\epsilon>0$, then $G$ has a subgroup $H$ such that $[G:H]$ and $|H'|$ are both $\epsilon$-bounded, that is, $G$ is bounded-by-abelian-by-bounded. Throughout, we say a group $G$ is $X$-by-$Y$ to mean that there is a normal $X$-subgroup $N$ in $G$ such that $G/N$ is a $Y$-group. We say ``$(a, b, c, \dots )$-bounded" to mean ``bounded by a function of $a, b, c, \dots$ only. Where the parameters $a, b, c, \dots$ are clear from the context we often just say ``bounded".  The main result of \cite{ebeshu} states that if $d_2(G)\geq\epsilon>0$, then $G$ has a subgroup $H$ such that $[G:H]$ and $|\gamma_4(H)|$ are both $\epsilon$-bounded. The structure of finite groups $G$ with $d_k(G)\geq\epsilon>0$, where $k\geq3$ remains a mystery.

Much of the paper \cite{ebeshu} consists of a study of groups satisfying a certain commutator covering condition. For a group $G$ and a positive integer $n$ write $B_n(G) = \{x \in G : |x^G | \le n\}$. It was shown in \cite{ebeshu} that if $d_2(G)\geq\epsilon>0$, then there are $\epsilon$-bounded positive integers $n$ and $s$ with the property that $G$ has a bounded index subgroup $K$ and a subset $S\subseteq K$ such that $|S|\leq s$ and $[x,y]\in B_n(K)S$ for any $x,y\in K$. Note that the condition makes sense for infinite as well as finite groups, and is satisfied by groups where all commutators have conjugacy classes of bounded size (the case S = \{1\}), which were studied in \cite{dieshu}. A theorem obtained in \cite{ebeshu} says that any group $K$ satisfying the above condition has a subgroup $T$ such that $[K:T]$ and $|\gamma_4(T)|$ are both finite and $(n,s)$-bounded. 

Given a group word $w=w(x_1,\dots,x_k)$, write $G_w$ for the set of all values $w(g_1,\ldots,g_k)$, where $g_1,\ldots,g_k$ are elements of $G$. We say that the group $G$ satisfies the $(n,s)$-covering condition with respect to the word $w$ if there exists a subset $S\subseteq G$ such that $|S| \le s$ and
 \begin{equation}\label{cover} G_{w}\subseteq B_n(G)S.\end{equation}

In the present paper we study groups $G$ admitting a word $w$ for which the set of $w$-values satisfies the above covering condition.  In particular, we work with multilinear commutator words, sometimes also called in the literature outer commutator words.   Examples of multilinear commutators include the lower central words $\gamma_k(x_1,\dots,x_k)=[x_1,\dots,x_k]$. Of course, $\gamma_k(G)= \langle G_{\gamma_k} \rangle$ 
  is the $k$-th  term of the lower central series of $G$.

We obtain the following results.

\begin{theorem}\label{soluble-main} Let $w$ be a multilinear commutator word on $k$ variables and let $G$ be a group satisfying the $(n,s)$-covering condition with respect to the word $w$. Then $G$ has a soluble subgroup $\T$ such that $[G : \T]$ and the derived length of $\T$ are both $(k,n,s)$-bounded.
\end{theorem}

\begin{theorem}\label{cov-nilp} Let $k\geq1$ and $G$ be a group satisfying the $(n,s)$-covering condition with respect to the word $\gamma_k$. Then
\begin{enumerate}
\item $\gamma_{2k-1}(G)$ has a subgroup $\T$ such that $[\gamma_{2k-1}(G) : \T]$ and $|\T'|$ are both $(k,n,s)$-bounded.
\item $G$ has a nilpotent subgroup $U$ such that $[G : U]$ and the nilpotency class of $U$ are both $(k,n,s)$-bounded.
\end{enumerate}
\end{theorem}

Note that in the cases where $G_w\subseteq B_n(G)$ the above theorems have been established in \cite{dms} and \cite{S-BFC}. All these results are natural generalizations of the theorem of B. H. Neumann saying that if $G$ is an $n$-BFC-group, that is $G=B_n(G)$, then $G'$ has finite $n$-bounded order \cite{bhn}.

\section{A technical result} 

We recall that multilinear commutator words  (multilinear commutators, for short), are recursively defined as follows:  The word $w(x)=x$ in one variable is a multilinear commutator; if $u$ and $v$ are  multilinear commutators involving disjoint sets of variables, then the word $w=[u,v]$ is a multilinear commutator, and all multilinear commutators are obtained in this way. Well-known examples of multilinear commutators are the aforementioned lower central words $\gamma_k$ and the derived words $\delta_i(x_1,\dots,x_{2^i}),$ defined by $\delta_0=x_1$ and \[\delta_{i+1}(x_1,\dots,x_{2^{i+1}})=
[\delta_i(x_1,\dots,x_{2^i}),\delta_i(x_{2^i+1},\dots,x_{2^{i+1}})].\]

The main result of this section is the following quite technical proposition, which is a stronger form  of Theorem 1.1 in  
\cite{dms}. As usual, $\langle X\rangle$ denotes the (sub)group generated by a set $X$.

 \begin{proposition}\label{pavels-lemma} Let $w=w(x_1,\dots,x_k)$ be a multilinear commutator, $G$ a group and $A_1$,\dots ,$A_k$ normal subgroups of $G$.  Let $B= B_n(G)$, $H=\langle B \rangle$, and $Y=\{w(g_1,\dots,g_n)  \,:\,  g_i\in A_i\}$.
 Assume that $Y\subseteq B.$ Then the  commutator subgroup  $[H,\langle Y\rangle]$ has finite $(n,k)$-bounded order.
 \end{proposition}


We will require combinatorial techniques developed in \cite{DMS-revised}.
 
Let $w=w(x_1, \dots, x_k)$ be a multilinear commutator. 
If $A_1,\dots,A_k$ are subsets of a group $G$, we write
 $w(A_1, \dots , A_k)$ for the subgroup generated by the  $w$-values $w(a_1,\dots,a_k)$ with $a_i\in A_i$.  
  
Let $I$ be a subset of $\{1,\dots,k\}$. 
  Suppose that we have a family $A_{i_1}, \dots , A_{i_s}$ of subsets of $G$ with indices running  over $I$ and another family 
  $B_{l_1}, \dots , B_{l_t}$ of subsets with indices  running  over $\{1, \dots ,k \} \setminus I.$ 
 We write 
 $$w_I(A_i ; B_i)$$ 
 for $w(X_1, \dots , X_k)$, where $X_j=A_j$ if $j \in I$, and $X_j=B_j$  otherwise. 
 On the other hand, whenever $a_i\in A_i$ for $i\in I$ and $b_i\in B_i$ for $i\in \{1,\dots,k\}\setminus I$, the symbol 
 $w_I(a_i;b_i)$ stands for the element $w(g_1, \dots , g_k)$, where $g_j=a_j$ if $j \in I$, and $g_j=b_j$ otherwise.  
 
The following results are  Lemma 2.4, Lemma 2.5  and Lemma 4.1 of \cite{DMS-revised}, respectively. 
 
\begin{lemma}\label{2.1-conjugates} 
Let $w=w(x_1, \dots, x_k)$ be a multilinear commutator. 
Assume that  $M$ is a normal subgroup of a group $G$. 
Let $ g_1, \dots , g_k \in G$, $h \in M$  and fix $l \in \{1, \dots, k\}$. 
Then there exist 
 $y_j \in g_j^M$,  for  $j=1,\dots, k$, such that 
 \begin{eqnarray*}
 w_{\{l\}}(g_lh; g_i)=w(y_1, \dots,y_k) w_{\{l\}}(h; g_i). 
\end{eqnarray*}
\end{lemma}

\begin{lemma}\label{uno2}
Let $w=w(x_1, \dots, x_k)$ be a multilinear commutator. Let $G$ be a group and let $A_1,\dots,A_k,M$ be  normal subgroups of $G$. Assume that $$w( y_1(A_1\cap M),\dots,y_k(A_k\cap M))=1$$ for some elements $y_i\in A_i$. Then 
$$w_I(y_i(A_i\cap M);A_i\cap M )=1,$$
for any subset $I$ of $\{1,\dots,k\}$. 
\end{lemma}

\begin{lemma}\label{M2} 
Let $w=w(x_1, \dots, x_k)$ be a multilinear commutator and let $I$ be a subset of $\{1, \dots ,k \}$.
Let $M$ be a normal subgroup of  a group $G$. 
 Assume that 
\[ w_J (A_j; A_j\cap M)=1 \quad \textrm{for every}\ J \subsetneq I.\]
Suppose we are given elements  $y_i \in A_i$ with $i \in I$ and  elements $m_r \in A_r\cap M$ with $r \in \{1, \dots, k\}$. 
 Then we have 
\[w_I(y_im_i; m_i)=w_I(y_i;m_i).\] 
\end{lemma}

Let $G$ be a group generated by a set $X$ such that $X = X^{-1}$. Given
an element $g\in G$, we write $l_X(g)$ for the minimal number $l$ with the
property that $g$ can be written as a product of $l$ elements of $X$. Clearly, $l_X(g)=0$ if and only if $g = 1$. We call $l_X(g)$ the length of $g$ with respect to $X$. The following result is Lemma 2.1 in \cite{dieshu}.

\begin{lemma}\label{2.1g}  Let $T$ be a group generated by a set $X = X^{-1}$ and let $M$ be a subgroup of finite index $n$ in $T$. Then each coset $Mb$ contains an element $g$ such that $l_X(g)\le n-1$.
\end{lemma}

Recall that if $G$ is a group, $a\in G$ and $H$ is a subgroup of $G$, then $[H,a]$ denotes the subgroup of $G$ generated by all commutators of the form $[h,a]$, where $h\in H$. It is well-known that $[H, a]$ is normalized by $a$ and $H$.

In the sequel we will assume the hypotheses of Proposition  \ref{pavels-lemma}; therefore
 $w=w(x_1, \dots, x_k)$ is a multilinear commutator, 
 $A_1$,\dots ,$A_k$ are normal subgroups of a group $G$, $H= \langle B \rangle$, where $B=B_n(G)$,  
  and every element of the set
 \[Y=\{w(g_1,\dots,g_n)  \,:\,  g_i\in A_i\}\]
 has at most $n$ conjugates in $G$. 
  
\begin{lemma} \label{2.3b} 
The subgroup $[H,y]^G$ has $n$-bounded order for every $y\in Y$.
\end{lemma}
\begin{proof}  Choose $y\in Y$.  Since $C_H (y)$ has index at most $n$ in $H=\langle B\rangle$, by Lemma \ref{2.1g} we can choose elements $h_1,\dots,h_n$ such that $l_B(h_i)\le n-1$ and $[H,y]$ is generated by the commutators $[h_{i},y]$. Note that, as  $l_B(h_i)\le n-1$, each $h_i$ has at most $n^{n-1}$ conjugates in $G$, so $C_H (h_i)$ has finite $n$-bounded index in $H$. As $C_H (y)$ has  index at most $n$ in $H,$ the centralizer $C$ of $ T = \langle y, h_1 , \dots, h_n\rangle$ in $H$ 
has finite $n$-bounded index in $H$.  Since $C \cap T$ is an $n$-bounded index subgroup contained in the center of $T$, 
 we deduce  that $Z(T )$ has  $n$-bounded index in $T$. Thus,
Schur's theorem \cite[10.1.4]{Rob} tells us that $T'$ has finite $n$-bounded order.
 As $[H, y]\le T'$,   it follows that $[H,y]$ has $n$-bounded order. 
 
 Observe that $[H,y]$ has at most $n$ conjugates in $G$ and the conjugates normalize each other. Thus, $[H,y]^G$ is a product of at most $n$ subgroups that normalize each other and have $n$-bounded order. The lemma follows.
\end{proof}

The following lemma can be seen as a generalization of Lemma 2.4 in \cite{dms}. It plays a central role in our arguments.  

\begin{lemma} \label{basic-light}  
Let $M$ be  a normal subgroup of $G$ of finite index $j$ and choose $a_i\in A_i$ for $i=1,\dots,k$. Then there exist elements $\tilde a_i\in a_i(M\cap A_i)$, for $i=1,\dots,k$, and a normal subgroup $\tilde M$ of $M$ 
of finite $(j,n)$-bounded index, such that the order of $[H,w(\tilde a_1 (A_1\cap\tilde M) ,\dots,\tilde a_k (A_k\cap\tilde M))]^G$ is finite and $n$-bounded.
\end{lemma}
\begin{proof}
Consider the set 
$$S=\{w( a_1u_1,\dots,a_ku_k)  \,:\,  u_i\in A_i\cap M\textrm{ for every } i \}.$$ 
Choose  $\tilde a_i\in a_i(M\cap A_i)$, for every $i$, such that   $a=w(\tilde a_1,\dots,\tilde a_k)$ has the maximum number of conjugates  in $H$ among all the elements of $S$, that is  
$|a^H|\ge |g^H|$ for all $g\in S$.
 
By Lemma \ref{2.1g}  we can choose $b_1,\dots, b_r\in H$ such that $l_B(b_i) \le n-1$ and $a^H = \{a^{b_i}  \,:\,  i = 1, \dots, r\}$. Let $\tilde M$ be the intersection of $M$ and all conjugates of $C_G (\langle b_1 ,\dots, b_r \rangle)$. Since $l_B(b_i) \le n-1$ and $C_G(x)$ has index at most $n$ in $G$ for each  $x \in B$, the subgroup  $C_G(\langle b_1, \dots , b_r\rangle)$ has $n$-bounded index in $G$, and so $\tilde M$ has $(j,n)$-bounded index. 

Consider the element  $w(\tilde a_1v_1,\dots,\tilde a_kv_k)\in S$ where $v_i\in (A_i\cap\tilde M)$, for $i=1,\dots,k$.  We have 
$$w(\tilde a_1v_1,\dots,\tilde a_kv_k)=va,$$ for some $v\in \tilde M\le  C_G ( b_1 ,\dots, b_r )$. 
It follows that $(va)^{b_i} = va^{b_i}$ for each $i =1,\dots, r$. Therefore the elements $va^{b_i}$ form the conjugacy class $(va)^H$ because they are all different and their number is the allowed maximum. So, for an arbitrary element $h\in H$ there exists $b\in\{b_1 ,\dots, b_r\}$ such that
$(va)^h= va^b$ and hence $v^h a^h = va^b$. Therefore $[h, v] = v^{-h}v=a^h a^{-b}$ and so $[h, v]^a =a^{-1} a^h a^{-b} a = [a,h][b,a] \in [H,a].$
 Thus $[H,v]^a \le [H,a]$ and $$[H, va]=[H,a] [H,v]^a \le [H, a].$$ 
This shows that $[H,w(\tilde a_1 (A_1\cap \tilde M),\dots,\tilde a_k (A_k\cap \tilde M))]\le [H,a]$.
Lemma \ref{2.3b} states that $[H,a]^G$ has $n$-bounded order. The result follows.
\end{proof}

For the reader's convenience, the most technical part of our argument is isolated in the following proposition. 

\begin{proposition}\label{inductive-step} 
Let $I\subseteq\{1,\dots,k\}$ with $|I|=s$. Under the hypotheses of Proposition \ref{pavels-lemma}, assume that $G$ contains a normal subgroup $R$ of finite order $r$ and  
a normal subgroup $M$ of finite index $j$ such that
\[[H, w_J (A_i; A_i\cap M)]\le R \quad \textrm{for every}\ J \subsetneq I.\]
Then there exists a finite normal subgroup $R_I$ of $G$ of $(r,j,n,s)$-bounded order with
$R\le R_I$ and  
a normal subgroup $M_I$ of $G$ of $(j,n,s)$-bounded index with $M_I\le M$ such that
\[ [H,w_I (A_i; A_i\cap M_I)]\le R_I.\]
\end{proposition}

\begin{proof}
 For every $i=1,\dots, k$ choose a left transversal $C_i$ of  $A_i\cap M$ in $A_i$. Observe that $|C_i| \le j$. 
  Let $\Omega$ be the set of $k$-tuples  $\underline{c}=(c_1, \dots , c_k)$ where $c_i \in C_i$ if $i\in I$ and $c_i=1$ otherwise. 
   Notice that the size of $\Omega$ is at most $j^s$. 
 For any  $k$-tuple  $\underline{c}=(c_1, \dots , c_k) \in \Omega$,  by  Lemma \ref{basic-light}, 
 there exist elements $d_i\in  c_i (A_i\cap M)$, with
 $i=1,\dots,k$, and a normal subgroup $M_{\underline c}$ of $(j,n)$-bounded index in $G$ such that the order of 
$$[H,w(d_1  (A_1\cap M_{\underline c}),\dots,d_k (A_k\cap M_{\underline c}))]^G$$
 is $n$-bounded.
Let 
\begin{eqnarray*}
M_I&=&M \cap \bigg( \bigcap_{\underline{c}\in \Omega}M_{\underline c}\bigg),\\
R_I&=& R \, \prod_{\underline{c}\in \Omega}[H,w(d_1(A_1\cap M_{\underline c}),\dots,d_k (A_k\cap M_{\underline c}))]^G.
\end{eqnarray*}
 As $|\Omega|\le j^s$,
 it follows that 
 $M_I$ has $(j,n,s)$-bounded index in $G$ and $R_I$ has $(r,j,n,s)$-bounded order.

Let $Z/R_I$ be the center of $HR_I/R_I$ in the quotient group $G/R_I$ and let $\bar G=G/Z$. The image of a subgroup
$U$ of $G$ in $\bar G$ will be denoted by $\bar U$, and similarly for the image of an element.

Pick an arbitrary element $w_I(a_i,h_i)\in w_I (A_i; A_i\cap M_i)$ (where $a_i \in A_i$ for $i\in I$,  and $h_i \in A_i\cap M_I$ otherwise).
Define the $k$-tuple  $\underline{c}=(c_1, \dots , c_k) \in \Omega$  by $a_i\in c_{i}(A_i\cap M)$ if $i\in I$ and $c_i=1$ otherwise. 
 Let $d_1,\dots,d_k$ be the elements as above, corresponding to the  $k$-tuple  $\underline{c}$. 
 Then 
$$[H,w(d_1(A_1\cap M_I),\dots,d_k (A_k \cap M_I))]\le R_I,$$
 that is 
$$\overline{w(d_1(A_1\cap  M_I),\dots,d_k(A_k\cap M_I))}=1,$$
 in the quotient group $\bar G=G/Z$. 
By Lemma \ref{uno2}, we deduce that 
\begin{equation}\label{step}
 \overline{w_I(d_i (A_i\cap M_I);A_i\cap M_i)}=1. 
\end{equation}
Moreover, as $c_{i}(A_i\cap M)=d_i(A_i\cap M)$, we have that $a_i=d_iv_i$ for some $v_i\in A_i\cap M$.
It also follows from our assumptions that  
$$\overline{w_J (A_i; A_i\cap M)}=1$$
for every proper subset $J$ of $I$. Thus we can apply Lemma \ref{M2} obtaining that 
$$w_I (\overline a_i;\overline h_i)=w_I (\overline d_i\overline v_i;\overline h_i)=
w_I (\overline d_i;\overline h_i)=1,$$
where in the last equality we have used (\ref{step}).
Since $w_I(a_i,h_i)$ was an arbitrary element of $ w_I (A_i; A_i\cap M_i)$, 
 it follows that $$\overline{w_I (A_i; A_i\cap M_i)}=1,$$
  that is \[ [H, w(A_i; A_i\cap M_i)]\le R_I,\]
as desired.
\end{proof}

Now the proof of Proposition \ref{pavels-lemma}  is an easy induction.

\begin{proof}[Proof of Proposition \ref{pavels-lemma}.] 
We will prove that for every $s=0,\dots,k$  the group $G$ contains a finite normal subgroup $R_s$ of $(n,k)$-bounded order and a normal subgroup $M_s$ of $(n,k)$-bounded index such that
\[ [H,w_I (A_i; A_i\cap M_s)]\le R_s\]
for every subset $I$ of $\{1,\dots,k\}$ with $|I|\le s$.
Once this is done, the theorem will follow taking $s=k$.

Assume that $s=0$. 
We apply Lemma \ref{basic-light} with $M=G$ and $a_i=1$ for every $i=1,\dots,k$. 
Thus  there exist elements $a_i\in A_i$ for each $i=1,\dots,k$ and  a normal subgroup of $n$-bounded index  $M_0$ of $G$, such that the order of 
$$R_0=[H,w( a_1(A_1\cap  M_0),\dots , a_k (A_k\cap M_0))]^G$$
 is $n$-bounded.

Let $Z/R_0$ be the center of $HR_0/R_0$ in the quotient group $G/R_0$ and let $\bar G=G/Z$.
Since 
$$\overline{w( a_1 (A_1\cap  M_0) ,\dots , a_k (A_k\cap  M_0))}=1,$$
it follows from Lemma \ref{uno2} that
$$\overline{w(A_1\cap  M_0,\dots ,A_k\cap  M_0)}=1,$$
 that is, $[H,w(A_1\cap  M_0,\dots ,A_k\cap  M_0)]\le R_0$. This proves the result in the case $s=0$.

Now assume  $s\ge 1$. Choose
$I\subseteq\{1,\dots,k\}$ with $|I|=s$. By induction, the hypotheses of Proposition \ref{inductive-step} are satisfied with $R=R_{s-1}$
and $M=M_{s-1}$, so there exists a finite normal subgroup $R_I$ of $G$ of $(n,k)$-bounded order with
$R_{s-1}\le R_I$ and  
a normal subgroup $M_I$ of $G$ of $(n,k)$-bounded index with $M_I\le M_{s-1}$ such that
\[ [H,w_I (A_i; (A_i\cap M_I)]\le R_I.\]

Let
 $$M_s=\bigcap_{|I|=s}M_I, \quad R_s=\prod_{|I|=s}R_I,$$
  where the intersection (resp. the product) ranges over all subsets $I$ of $\{1,\dots,k\}$ of size $s$.

As there is a $k$-bounded number of choices for $I$, it follows that $R_s$ (resp. $M_s$) has  $(n,k)$-bounded order (resp. $(n,k)$-bounded index 
in $G$).  Note that $M_s\le M_{s-1}$ and $R_{s-1}\le R_s$.
Therefore
\[ [H,w_I (A_i; A_i\cap M_s)]\le R_s\] 
 for every $I\subseteq\{1,\dots,k\}$ with $|I|\le s$. 
This completes the induction step. The proposition is established.
\end{proof}

\section{Proof of Theorem \ref{soluble-main}} 

Throughout this section we will assume that $w=w(x_1,\dots,x_k)$ is a multilinear commutator word and the group $G$ satisfies the  $(n,s)$-covering condition:
\begin{equation}\label{eq:cover}G_w\subseteq BS,\end{equation}
where 
 $B = B_n(G)= \{x \in G  \,:\,  |x^G | \le n\}$, and
$S$ is a subset of $G$ of size $s$. We will show that $G$ has a soluble subgroup $\T$ such that $[G : \T]$ and the derived length of $\T$ are both 
$(k,n,s)$-bounded. Actually, $T$ can be chosen of derived length at most $2k+1$. 

We will use the following key property of multilinear commutators.

\begin{lemma}\label{lem:deltak} Let $w=w(x_1,\dots,x_k)$ be a multilinear commutator and let 
$\delta_{k-1}$  
 be the $(k-1)$-th derived word. Then  $G_{\delta_{k-1}}\subseteq G_w$ for every group $G$.
\end{lemma}
\begin{proof} The proof is by induction on $k$. If $k=1$ then $w=\delta_0=x$ and the result is true. So now assume that $k\ge 2$.
Then $w=[v_1,v_2]$ where $v_1,v_2$ are multilinear commutators on $k_1$ and $k_2$ variables respectively, 
 where clearly $k_1, k_2 \le k-1$. By induction,  every $\delta_{k_i-1}$-value is a $v_i$-value. In particular, every $\delta_{k-2}$-value is a $v_i$-value, for $i=1,2$, 
  and the result follows. 
\end{proof}

At some point, without loss of generality, we will be able to assume that $G$ satisfies the following ``stability condition".
 \begin{equation}\label{stab1}\text{Any $w$-value $x$ such that  $|x^G | \le n^{2^k}$ is contained in $B$.} \end{equation}
 
\begin{lemma}\label{step1}  Suppose $G$ satisfies the condition \eqref{stab1}. Let $H= \langle B\rangle$. Then  
  $w(H, G, \dots, G)$ has derived subgroup of finite $(n, k)$-bounded order.
\end{lemma}
\begin{proof}
For any element $h \in H$, let $l_{B}(h)$ be the minimal number $j$ such that $h$ can be written as a product of $j$ elements from $B$. Clearly $|h^G| \le n^{l_{B}(h)}$.   

Assume there is an element $h\in H$ such that $w(h, g_2, \dots , g_k) \notin B$ for some $g_i\in G$ and choose $h$ in such a way that $l_{B}(h)$ is minimal. Note that, if  $h \in B$, 
then the element $ w(h, g_2, \dots , g_k)$ can be written as a product of $2^{k-1}$ conjugates of the elements $h$, $h^{-1}$, whence,  
by the stability condition \eqref{stab1}, $ w(h, g_2, \dots , g_k) \in B$, a contradiction. 

 Thus, $l_{B}(h) \ge 2$ and  we can write $h=h_1 h_2$ where $l_{B}(h_1), l_{B}(h_2) < l_{B}(h)$. By Lemma \ref{2.1-conjugates} 
  there exists $y_1 \in h_1^G$ and 
 $y_j \in g_j^G$,  for  $j=2,\dots, k$, such that 
\[  w(h_1h_2, g_2, \dots , g_k)= w(y_1,y_2, \dots , y_k) w(h_2, g_2,  \dots , g_k).
\]
By minimality, both factors on the right hand side of the above equality belong to $B$. So $w(h, g_2, \dots , g_k)$ has at most $n^2$ conjugates, and therefore, by the stability condition \eqref{stab1}, it belongs to $B$, a contradiction. 
  
  This proves that for every $h \in H$ and every $g_i \in G$ the element  $w(h, g_2, \dots , g_k) $ belongs to $B$. So we can apply Proposition \ref{pavels-lemma} 
   to deduce that $w(H, G, \dots, G)'$ has finite  $(n, k)$-bounded order. 
\end{proof}

\begin{proof}[Proof of Theorem \ref{soluble-main}]
Note that it is sufficient to prove that $G$ has a subgroup $\tilde T$ such that $[G : \tilde T]$ and $|\tilde T^{(2k)} |$ are both $(k,n,s)$-bounded, since then 
 $C_{\tilde T}(\tilde T^{(2k)})$ will be a soluble subgroup of derived length at most $2k+1$ and $(k,n,s)$-bounded index in $G$. 

Without loss of generality we may assume that $1\in S$. In what follows we argue by induction on $s$. 
If $s=1$, then $G_w\subseteq B$ and it follows from Theorem 1.2 in \cite{dms} that $w(G)'$ has finite $(k,n)$-bounded order. Since $G^{(k-1)}\le w(G)$ by Lemma \ref{lem:deltak}, we conclude that $G^{(k)}$ has finite $(k,n)$-bounded order. 

So assume $s\ge 2$ and let $A= 2^k$. 
Suppose there is some commutator $x\in G_w$ such
that $n < |x^G | \le n^{A}$. By the covering condition $x\in Ba$ for some nontrivial $a \in S$.
This implies that $|a^G |\le n^{A+1}$. Then for any $g \in Ba$ we have $|g^G |\le n^{A+2}$. Hence we
can remove $a$ from $S$ at the cost of increasing $n$ to $n^{A+2}$ and then apply induction
on $s$. Hence we may assume $B$ satisfies the  stability condition  \eqref{stab1}:
Any $w$-value $x$ such that  $|x^G | \le n^{2^k}$ is contained in $B$.

 Let  $H = \langle B\rangle$. 
 It follows from  Lemma \ref{step1} that  $w(H, G, \dots, G)$ has commutator subgroup of finite 
  $(k,n,s)$-bounded order. Note that $H^{(k-1)} \le w(H)$ by Lemma \ref{lem:deltak}.
Hence, $H^{(k)} \le w(H)' \le w(H, G, \dots, G)'$ and we deduce that $H^{(k)}$ has  finite $(k, n, s)$-bounded order as well.  
  
Note that $H$ is a normal subgroup of $G$. Taking into account the covering condition observe that $w$ takes at most $s$ values in $G/H$. As $w$ is a boundedly concise word  (see \cite{fernandez-morigi}), it follows that the verbal subgroup $w(G/H)$ has $(s,k)$-bounded order. So the centralizer $C/H=C_{G/H}(w(G/H))$ has $(s,k)$-bounded index. Moreover $G^{(k-1)}\le w(G)$ by Lemma \ref{lem:deltak}, whence  $C^{(k)}\le H$. 
 Therefore $C^{(2k)} \le H^{(k)}$ has  finite $(k, n, s)$-bounded order. This completes the proof. 
\end{proof}


\section{Proof of Theorem \ref{cov-nilp}}  

The goal of this section is to furnish a proof of Theorem \ref{cov-nilp}. We start with some general facts that will be used in the proof.

If $K$ is a subgroup of a finite group $G$, the commuting probability $\pr (K,G)$ of $K$ in $G$ is the probability that a random element of $K$ commutes with a random element of $G$. Thus, 
\[ \pr(K,G) =\frac{ | \{ (x,y) \in K\times G \,:\, xy=yx \} |}{|K|\,|G|}.\]
 Note that if $|x^G|\leq n$ for every $x\in K$, then $\pr (K,G)\geq {1}/{n}$. It was shown in \cite{DS}
  that if $\pr (K,G)\geq\epsilon>0$, then there is a normal subgroup $T\leq G$ and a subgroup $B\leq K$ such that the indices $[G:T]$ and $[K:B]$ and the order of the commutator subgroup $[T,B]$ are $\epsilon$-bounded. 
  Moreover, if $K$ is normal in $G$, then we can assume that $B$ is normal in $G$ and $K \cap T \le Z_3(T)$ (see \cite[Remark 2.7]{DS}). 

We recall the following well-known criterion for nilpotency of a group. The result is due to P. Hall \cite{hall1}.
\begin{theorem}\label{hall}
Let $N$ be a nilpotent normal subgroup of a group $G$ such that $N$ and $G/N'$ are both nilpotent. Then $G$ is nilpotent of class bounded in terms of the classes of $N$ and $G/N'$.
\end{theorem}

The following results are Lemma 2.2 and Lemma 2.4 in \cite{S-BFC}.

\begin{lemma}\label{redufini}
Let $G$ be a group and $k,m$ positive integers.
\begin{enumerate}
\item $|\gamma_k(G)|\leq m$ if and only if $|\gamma_k(H)|\leq m$ for any finitely generated subgroup $H\leq G$.
\item If $G$ is residually finite, $|\gamma_k(G)|\leq m$ if and only if $|\gamma_k(Q)|\leq m$ for any finite quotient $Q$ of $G$.
\end{enumerate}
\end{lemma}

\begin{lemma}\label{lll}
Let $G$ be a metabelian group and suppose that $a,b\in G$ are $l$-Engel elements. If $x\in\langle a,b\rangle$, then $x$ is $(2l+1)$-Engel.
\end{lemma}
 We say for short that a group $A$, acting on a group $G$ by automorphisms, acts   coprimely on $G$ if the orders of $A$ and $G$ are
finite and coprime, that is, $(|A|, |G|) = 1$. The following lemma records a standard result on coprime action (see \cite[(24.5)]{asch}).

\begin{lemma}\label{coprime}
Let $A$ be a finite group acting coprimely on a finite group $G$ such that $(|A|,|G|)=1$. Then $[G,A,A]=[G,A]$.
\end{lemma}

Throughout the remaining part of the paper we assume the group $G$ satisfies the $(n,s)$-covering condition with respect to the word $\gamma_k$:
\begin{equation*}G_{\gamma_k}\subseteq BS,\end{equation*}
where $B =B_n(G)= \{x \in G  \,:\,  |x^G | \le n\}$ and $S$ is a finite subset of $G$ of size $s$. 

First, we prove that, under the  above hypothesis, 
 $\gamma_{2k-1}(G)$ 
  has a  subgroup $\T$ such that $[\gamma_{2k-1}(G) : \T]$ and $|\T'|$ are both $(k , n, s)$-bounded.
	  
\begin{proof}[Proof of Theorem \ref{cov-nilp} (1).] 
We may assume $1 \in S$. If $s=1$, then  the result holds
 by Theorem 1.2 in \cite{dms}, since  $\gamma_k(G)'$ turns out to have finite $(n,k)$-bounded order.  
 
 Assume $s \ge 2$ and argue by induction on $s$.  Suppose there is a commutator $x\in G_{\gamma_k}$ such
that $n < |x^G | \le n^{2^k}$. By the covering condition $x\in Ba$ for some nontrivial $a \in S$.
This implies that $|a^G |\le n^{2^k+1}$. Then for any $g \in Ba$ we have $|g^G |\le n^{2^k+2}$. Hence we
can remove $a$ from $S$ at the cost of increasing $n$ to $n^{2^k+2}$ and then apply induction
on $s$. 
Hence, we may assume $B$ satisfies the stability condition \eqref{stab1}: 
\begin{equation*}\label{stab-gamma}\text{Any $\gamma_k$-value $x$ such that  $|x^G | \le n^{2^k}$ is contained in $B$.}
\end{equation*}
Let $H=\langle B \rangle$. It follows from Lemma \ref{step1} that  $[H, G, \dots, G]'$ has finite $(k,n)$-bounded order. 

Now it is sufficient to prove that $[H, G, \dots, G]$ has bounded index in $\gamma_{2k-1}(G)$. 
 
 Note that in $G/H$ there are at most $s$ values of the word $\gamma_k$. As $\gamma_k$ is boundedly concise  (see \cite{fernandez-morigi}), we deduce that $\gamma_k (G/H)$ has finite bounded order. Therefore the index of $H\cap \gamma_k (G)$ in $\gamma_k (G)$ is bounded. 
  
   Let $\overline G$ be the quotient group of $G$ over $[H, G, \dots, G]$. The image $\overline H$ of $H$ in $\overline G$ is contained in $Z_{k-1}(\overline G).$  
   Here and throughout $Z_i(M)$ denotes the $i$th term of the upper central series of a group $M$. 
   Since the index of $\overline H\cap \gamma_k (\overline G)$ in $\gamma_k (\overline G)$ is bounded, we deduce that 
 $Z_{k-1}(\overline G) \cap \gamma_k (\overline G)$ has bounded index in $\gamma_k (\overline G)$. 
 
It follows from Theorem B in \cite{guma} that $Z_{2{k-1}}(\overline G) $ has bounded index in $\overline G$. Hence, we deduce from Baer's theorem (see \cite[4.5.1]{Rob}) that $\gamma_{2(k-1)+1} (\overline G)$ has finite bounded order. Therefore $[H, G, \dots, G]$ has bounded index in $\gamma_{2k-1} ( G)$, as claimed, and the proof is complete. 
\end{proof}

Now we proceed with the proof of the second part of Theorem \ref{cov-nilp}. We need to show that $G$ is nilpotent-of-bounded-class-by-bounded,
with bounds depending only on $s$, $n$ and $k$.

\begin{proof}[Proof of Theorem \ref{cov-nilp} (2)]
By Theorem \ref{soluble-main}, we can assume that $G$ is soluble of bounded derived length. 

 By induction on the derived length, $G'$ has a nilpotent subgroup $A$ of bounded index
and class. In view of \cite{khuma} we can assume that $A$ is characteristic in $G'$. It follows that $A$ is normal
in $G$. If $A'\ne 1$ then
by induction on the class of $A$ there is a nilpotent subgroup of $G/A'$ of bounded
index and class. By Hall's criterion we are done in this case, so we may assume that
$A$ is abelian. Since $G'/A$ is bounded, we may replace $G$ with a bounded-index
subgroup whose image in $G/A$ has class at most $2$.
Thus, without loss of generality we assume that $G$ is abelian-by-class-2. By another classical result of Hall, any
abelian-by-nilpotent group is locally residually finite (see \cite[15.4.1]{Rob}).Thus, by 
Lemma \ref{redufini}, it is sufficient to prove the result for finite quotients of finitely generated subgroups of $G$. Therefore without loss of generality we may assume that $G$ is finite and $\gamma_3(G)$ is abelian. 

 Assume first that $G$ is metabelian. Let $N=G'$. Set 
\[{\bf G}=G\times\dots\times G\ \ (k-1\text{ factors}).\]
 If ${\bf g}=(g_1,\dots,g_{k-1})\in{\bf G}$, let $N_{\bf g}=[N,g_1,\dots,g_{k-1}]$. 
  As $N$ is abelian,
   the map $a\mapsto [a,g_1,\dots,g_{k-1}]$ is a
homomorphism $N\to N_g,$ thus the subgroup $N_{\bf g}$ consists of $\gamma_k$-values, 
 hence $N_{\bf g}\subseteq BS$. 

 Suppose that $[a_1,g_1,\dots,g_{k-1}],[a_2,g_1,\dots,g_{k-1}]\in Bx$ for some $a_1,a_2\in N$ and $x\in S$. Then
$$[a_1a_2^{-1},g_1,\dots,g_{k-1}]=[a_1,g_1,\dots,g_{k-1}][a_2,g_1,\dots,g_{k-1}]^{-1}\in B^2.$$

As $|S| \le s$, 
it  follows that
$$\pr _{ a\in N} ([a,g_1,\dots,g_{k-1}]\in B^2 )\ge 1/s,$$
so
$$\pr _{ a\in N\, h\in G} ([a,g_1,\dots,g_{k-1}, h] = 1) \ge\frac 1{sn^2},$$
that is, $\pr (N_{\bf g},G)\geq\frac{1}{s n^2}$. As $N=G'$ is abelian and $g_i^g\in Ng_i$ for each $g\in G$ and for each $i$, it follows that $N_{\bf g}$ is a normal subgroup of $G$. 
By Proposition 1.2 and Remarks 2.6 - 2.7 in \cite{DS}
  there are two  normal subgroups $T_{\bf g}$ and  $B_{\bf g}\leq N_{\bf g}$ such that the indices $[G:T_{\bf g}]$ and $[N_{\bf g}:B_{\bf g}]$
   and the order of $[T_{\bf g},B_{\bf g}]$ are $(n,s)$-bounded. Moreover, the subgroup $T_{\bf g}$ can be chosen in such a way  that $N_{\bf g} \cap T_{\bf g} \le Z_3(T_{\bf g})$. 

We now will handle the particular case where $G$ is a metabelian $p$-group. 

As $[T_{\bf g},B_{\bf g}]$
 is a normal subgroup of  $n$-bounded order, it follows that there is an $n$-bounded number $e$ such that $Z_e(G)$ contains $[T_{\bf g},B_{\bf g}]$ for every ${\bf g}\in {\bf G}$. 
Pass to the quotient $G/Z_e(G)$ and assume that $[T_{\bf g},B_{\bf g}]=1$ for any ${\bf g}\in {\bf G}$. 

We now prove the following claim: 
\begin{itemize}
\item[(*)] {\it Suppose that $G$ is $l$-Engel for some $l\geq1$. Then $G$ is nilpotent of $(k,l,n)$-bounded class.}
\end{itemize}
 Since $G/T_{\bf g}$ has bounded order, $G =\langle x_1,\dots,x_r,T_{\bf g}\rangle$ for some elements $x_1,\dots,x_r\in G$
and bounded $r$. Moreover $G$ is metabelian and $l$-Engel, so $N\langle x_i\rangle$ has class at most $l$
for each $i$, and therefore $N\langle x_1,\dots,x_r\rangle$ has nilpotency class at most $lr$ by Fitting's theorem. In particular
$B_{\bf g}\langle x_1,\dots,x_r\rangle$ has bounded class. It follows from $[T_{\bf g},B_{\bf g}]=1$ that $B_{\bf g}$ is contained in $Z_j(G)$ for some bounded number $j$. Moreover $N_{\bf g}/B_{\bf g}$ has bounded order, so $N_{\bf g}$ is
contained in $Z_{j'}(G)$ for some bounded number $j'$. Thus $[N,g_1,\dots,g_{k-1}]$
is contained in $Z_{j'}(G)$ for all $g_1,\dots,g_{k-1}\in G$.  As $N=G'$, we conclude that 
$\gamma_{k+1}(G)\le Z_{j'}(G)$ and so $G$  is nilpotent of class at most $k+j'$. 
Hence, the claim follows. 

Thus, it suffices to find a bounded-index subgroup of $G$ which is $l$-Engel for some bounded $l$.
We will use the notation $[x,{}_iy]$ to denote the long commutator $[x,y,\dots,y]$, where $y$ is repeated $i$ times.

If $g\in G$, for short we will write 
\[{\bf \tilde g}=(g,\dots,g)\in{\bf G},\]
so $N_{\bf \tilde g}=[N,{}_{k-1}g]$.
Since $N_{\bf \tilde g}/B_{\bf\tilde g}$ is a $p$-group of bounded order, there is a bounded number $f$ such that
$$[N_{\bf\tilde g},\;_f g]=[N,{}_{k-1+f}g] \le B_{\bf\tilde g}$$ for all $g\in G$. For $g\in G$ and $i\ge k-1$ write $B_{i,g}$ for $[B_{\bf\tilde g},{}_{i-k+1}\,g]$. As $B_{\bf\tilde g}\le N$ is abelian and $g^y\in Ng$ for each $y\in G$, it follows that $B_{i,g}$ is normal in $G$.
  Put $C_{i,g}=C_G(B_{ig})$. 
Let $\beta_i$ be the maximal index of $C_{i,g}$, where $g$ ranges over $G$. Since $[T_{\bf \tilde g},B_{\bf\tilde g}]=1$, it follows that $\beta_{k-1}$ is ($|S|$,$n$)-bounded. It is clear that $\beta_{k-1}\geq\beta_k\geq\beta_{k+1}\geq\dots$, thus there is a bounded number $u$ such that $k-1\leq u$ and $\beta_u=\beta_{2u+f+k-1}$.

Choose $g\in G$ such that $C_{u,g}=C_{2u+f+k-1,g}$ is of index $\beta_u$ in $G$. We have $C_{u,g}=C_{j,g}$ for $j=u,u+1,\dots,2u+f+k-1$. Let $h\in C_{u,g}$. Note that 
\begin{eqnarray*}
B_{2u+f+k-1,g} &=&[B_{u,g},{}_{u+f+k-1}g]=[B_{u,g},{}_{u+f+k-1}\,(gh)]\\
&\le& [N,{}_{u+f+k-1}\,(gh)] \le[B_{\bf\tilde{gh}},{}_{u}\,(gh)] \\
&\le&
B_{u,gh}.
\end{eqnarray*}
So $C_{u,gh}\le C_{2uf+k-1,g}=C_{u,g}$. 
Since $\beta_u$ is maximal and since $h$ was chosen in $C_{u,g}$ arbitrarily, we are forced to conclude that $$C_{u,g}=C_{u,gh} \text{ whenever } h\in C_{u,g}.$$ Set $C_{u,g}=C_0$ and note that $C_0$ centralizes $B_{u,gh}$ for any $h\in C_0$. 

Obviously, $C_0$ is a normal subgroup of index $\beta_u$ in $G$. Let $\overline{G}=G/Z(C_0)$. Since $Z(C_0)$ contains $B_{u, gh}=[B_{\tilde{\bf gh}},{}_{u-k+1}\,(gh)]$ and also  $[N,{}_{k-1+f}\,(gh)] \le B_{\bf\tilde{gh}}$ we conclude that $\overline{gh}$ is $(u+f+1)$-Engel in $\overline{G}$ for any $h\in C_0$. We deduce from Lemma \ref{lll} that $C_0$ is $l$-Engel for $l=2u+2f+3$. 
Now 
  it follows from the above claim $(*)$ 
 that $C_0$ is nilpotent of bounded class and so in the case where $G$ is a metabelian $p$-group the result follows. 

Now we deal with the particular case where $G/N$ is an abelian $p$-group while $N$ is a $p'$-group. 

 Since $N_{\bf g} \cap T_{\bf g} \le Z_3(T_{\bf g})$, we have that $[[N_{\bf g} ,T_{\bf g}] ,T_{\bf g} ,T_{\bf g}, T_{\bf g}] =1$. It follows from Lemma \ref{coprime} that $T_{\bf g}$ centralizes $N_{\bf g}$.

So we simply assume that $B_{\bf g}=N_{\bf g}$ and $T_{\bf g}=C_G(N_{\bf g})$. Note that if $g\in G$ and ${\bf g}=(g,\dots,g)\in {\bf G}$, then $N_{\bf g}=[N,g]$ so in what follows we write $N_g$ in place of $N_{\bf g}$ and $T_g$ in place of $T_{\bf g}$.

Choose $g\in G$ such that the index $[G:T_g]$ is maximal. Let $h\in T_g$. Because of Lemma \ref{coprime} we have $$N_g=[N_g,g]=[N_g,gh]\leq N_{gh}.$$ It follows that $T_{gh}\leq T_g$ for any $h\in T_g$. Since $[G:T_g]$ is maximal, deduce that $T_{gh}=T_g$ for any $h\in T_{ g}$. So $T_g$ centralizes both $N_g$ and $N_{gh}$. It is easy to see that $N_h\leq N_gN_{gh}$, so $T_g$ centralizes $N_h$. This holds for every $h\in T_g$ and so $[N,T_g,T_g]=1$. In view of Lemma \ref{coprime} deduce that $[N,T_g]=1$. Therefore $T_g$ is nilpotent of class at most $2$
and this proves the theorem in the coprime metabelian case. 

We now handle the case where $G$ is metabelian without making any assumptions on the order of $N$ and $G/N$. Let $e=n!$. As $G$ acts on $B$ by conjugation and each orbit has size at most $n$, it follows that  $G^e$ centralizes $B$. So $\langle B\rangle\cap G^e\le Z(G^e)$. The covering condition implies that the number of $\gamma_k$-values in the group $G^e/Z(G^e)$ is at
most $s$. Taking into account that $\gamma_k$ is a boundedly concise word
it follows that $\gamma_k\left(G^e/Z(G^e)\right)$ has bounded order, and its centralizer in $G^e$ is a bounded-index subgroup of nilpotency class at most $k+1$.  
 For each prime $p\leq n$ let $G_p$ be the preimage in $G$ of the Sylow $p$-subgroup of $G/N$. For $p>n$, the Sylow $p$-subgroups of $G$ are contained in $G^e$. Hence $G=G^e\prod_{p\leq n}G_p$. Therefore by Fitting's theorem it suffices to prove the result for $G=G_p$. Thus we may assume that $G/N$ is a $p$-group. Let $P$ be a Sylow $p$-subgroup of $G$. Then $G=NP$. By the result on $p$-groups, we can replace $P$ with a subgroup of bounded index and bounded class, so without loss of generality we assume that $P$ has $(k,n)$-bounded class $c$, say. It follows that $[P\cap N,{}_cG]=1$ and  so $P\cap N\leq Z_c(G)$. Factoring out $Z_c(G)$, we may thus assume that $N$ is a $p'$-group. By the result in the coprime case we are done.

We now drop the assumption that $G$ is metabelian and consider the general case. As $\gamma_3(G)$ is abelian, $G'$ is metabelian. Therefore, by the above,
 $G'$ has a nilpotent subgroup $A$ of bounded index and  bounded class. In view of \cite{khuma} we can assume that $A$ is characteristic in $G'$ and in particular normal in $G$. Applying Theorem \ref{hall} we reduce to the case where $A$ is abelian. 

If $A=G'$, then $G$ is metabelian and the result holds. We therefore assume that $A<G'$. Further, replacing $G$ by $C_G(G'/A)$ we assume that $G/A$ is nilpotent of class at most $2$. 
In particular, $\gamma_2(G') \le A$. 
 Fix a set $X$ of size at most $|G'/A|$ such that $G'$ is generated by $A$ and $X$. Set $${\bf X}=X\times\dots\times X\ \ (k-1\text{ factors}).$$ If ${\bf x}=(x_1,\dots,x_{k-1})\in{\bf X}$, the subgroup
 $A_{\bf x}=[A,x_1,\dots,x_{k-1}]$ 
 consists of $\gamma_k$-values. 
 Hence, arguing as above, we deduce that $\pr (A_{\bf x},G)\geq\frac{1}{s n^2}$.  
As $G$ centralizes $G'/A$, we have that $x_i^g\in Ax_i$ for each $i$, and because of the fact that $A$ is a normal abelian subgroup of $G$ it follows that $A_{\bf x}$ is normal in $G$.
 By \cite[Remark 2.7]{DS}  there is a normal subgroup $T_{\bf x}$ and a normal subgroup $B_{\bf x}\leq A_{\bf x}$ such that the indices $[G:T_{\bf x}]$ and $[A_{\bf x}:B_{\bf x}]$ are $(s, n)$-bounded and 
  $T_{\bf x} \cap A_{\bf x} \le Z_3(T_{\bf x})$. 
  Set 
  \[H= \bigcap_{{\bf x} \in{\bf X}} T_{\bf x}. \] 
  Since $|{\bf X}|\leq|G'/A|^{k-1}$, observe that the index of $H$ in $G$ is bounded. 
 Therefore it is sufficient to show that $H$ has a nilpotent subgroup of bounded index and bounded class. Note that $A_{\bf x} \cap H \leq Z_3(H)$ for any ${\bf x}\in{\bf X}$. 
  Since $\gamma_2(G') \le A$, we have $\gamma_2(H') \le A \cap H' \le G'=A\langle X\rangle.$ Thus every $\gamma_{2+k-1}$-value of $H'$ 
   can be written as a product of elements each of whom belongs to 
   $A_{\bf x}\leq Z_3(H)$, for some ${\bf x} \in{\bf X}$. 
 Therefore $\gamma_{2+k-1+3}(H')=1$ and 
  $H'$ is nilpotent of class at most $k+3$. An application of Hall's Theorem \ref{hall} reduces the problem to the case where $H'$ is abelian. We already know that for metabelian groups the theorem holds. This completes the proof of Theorem  \ref{cov-nilp}. 
\end{proof}

\end{document}